\documentclass[11pt]{amsart}
\usepackage{amsmath}
\usepackage{amsthm}
\usepackage{latexsym}
\usepackage{amssymb}
\usepackage{xspace}
\usepackage{amscd}
\usepackage{enumerate}

\newcommand{\Cen}{\mbox{\rm Cen}}

\newcommand{\inv}{^{-1}}

\newcommand{\Z}{{\mathbb Z}}
\newcommand{\Q}{{\mathbb Q}}
\newcommand{\R}{{\mathbb R}}
\newcommand{\C}{{\mathbb C}}
\newcommand{\HQ}{{\mathbb H}}

\newcommand{\GL}{{\rm GL}}
\newcommand{\SL}{{\rm SL}}

\newcommand{\matriz}[1]{\begin{array} #1 \end{array}}

\newcommand{\GEN}[1]{\langle #1 \rangle}

\newcommand{\quat}[2]{\left( \frac{#1}{#2} \right)}

\title{Subgroup separability in integral group rings}

\author{\'{A}ngel del R\'{\i}o}
\address{Departamento de Matem\'{a}ticas, Universidad de Murcia, Murcia 30100, Spain}
\email{adelrio@um.es}

\author{Manuel Ruiz Mar\'{\i}n}
\address{Departamento de M\'{e}todos Cuantitativos e Inform\'{a}ticos, Universidad Polit\'{e}cnica de Cartagena, C/ Real 3, 30201 Cartagena, Spain}
\email{manuel.ruiz@upct.es}

\author{Pavel Zalesski}
\address{Departamento de Matem\'{a}tica, Universidade de Bras\'{\i}lia, 70.910-900, Brasilia-DF, Brasil}
\email{pz@mat.unb.br}

\thanks{The first and second authors has been partially supported by the Ministerio de Ciencia y Tecnolog\'{\i}a of Spain and Fundaci\'{o}n S\'{e}neca of Murcia. The third author is partially supported by CNPq.}


\date{}

\newtheorem{theorem}{Theorem}
\newtheorem{lemma}[theorem]{Lemma}
\newtheorem{proposition}[theorem]{Proposition}

\theoremstyle{remark}
\theoremstyle{remark}

\begin{document}

\begin{abstract}
We give a list of finite groups containing all finite groups
$G$ such that the group of units $\Z G^*$ of the integral group ring $\Z G$ is subgroup separable.
There are only two types of these groups $G$ for which we cannot decide wether $ZG^*$ is subgroup
separable, namely the central product $Q_8 Y D_8$ and $Q_8\times C_p \mbox{ with } p \text{ prime and } p\equiv -1 \mod
(8)$.
\end{abstract}

\maketitle


A group  $\Gamma$ is said to be subgroup separable if for every
finitely generated subgroup $H$ of $\Gamma$ and  $g\in
\Gamma\setminus H$ there exists a subgroup of finite index $K$ of
$\Gamma$ such that $g\not\in KH$. In other words $\Gamma$ is
subgroup separable if every finitely generated subgroup of
$\Gamma$ is closed in the profinite topology of $\Gamma$ (i.e. the
topology generated by normal subgroups of finite index). The
importance of subgroup separability have long been recognized,
both in group theory and topology. This powerful property has
attracted a good deal of attention in the last few years, largely
motivated by questions which arise in low dimensional topology
(see \cite{scott}, and \cite{ALR} for example). The
first author who observed the importance of the subgroup separability
property was Mal'cev: he noticed that a subgroup separable
finitely presented group has solvable generalized word problem. It
is clear that subgroup separability of a group indicates that
its profinite topology  is strong. For arithmetic groups the
meaning of the profinite topology being strong is defined
concretely by means of the congruence subgroup property. It is
known that the congruence subgroup property for non-polycyclic
arithmetic groups implies non subgroup separability.

 There are few examples of
non-abelian groups that are known  to be subgroup separable. We
give a list of arithmetic groups known to have this property, since
it is relevant to the subject of this paper.  M. Hall \cite{hall}
provided the first non-trivial examples by proving that free
groups are subgroup separable. R. G. Burns \cite{burns} and N. S.
Romanovskii \cite{roma} showed that a free product of subgroup
separable groups is subgroup separable. These results were all
proved using algebraic methods. A more topological approach was
developed by J. Hempel in \cite{hempel}, J. R. Stallings in
\cite{stall} and P. Scott in \cite{scott}. Scott used hyperbolic
geometry to prove that surface groups are subgroup separable. More
recently, D. Long and A. Reid \cite{reid} adapted Scott's approach
to show that geometrically finite subgroups of certain hyperbolic
Coxeter groups are subgroup separable. In fact a combination of
the Agol, Long and Reid results \cite{A,ALR} proves
subgroup separability of Bianchi groups (see Theorem 3.4
in\cite{LR}) and so for all non-uniform arithmetic lattices.

In this paper we consider the problem of classifying  finite
groups $G$ such that $\Z G^*$, the group of units of the integral
group ring $\Z G$, is subgroup separable. To this end, we first
prove that $\Z G^*$ is subgroup separable if and only if the
simple components of the rational group algebra $\Q G$ satisfy
some special conditions. To classify the finite groups $G$ with
such rational group algebra we use firstly some representation
theory techniques and secondly some results of Jespers and Leal
\cite{JL1,JL2} and Gow and Huppert \cite{GH,GH2} on simple
components of rational group algebras. Throughout the paper we
will need to compute the Wedderburn decomposition of $\Q G$ for
some finite groups $G$. The reader can check these computations
using a method introduced in \cite{ORS} or the GAP package
Wedderga \cite{GAP,wedderga}.

\bigskip


We start introducing the basic notation.

The group of units of a ring $R$ is denoted $R^*$.
We will use $\zeta_n$ to denote a complex primitive $n$-th root of unity.

The commutator subgroup of a group $G$ is denoted $G'$. If $x,y\in G$ then $x^y=y\inv xy$ and $(x,y)=x\inv y\inv xy$.
The cyclic group of order $n$ is denoted $C_n$. We also use $\GEN{x}_n$ to denote a cyclic group of order $n$ generated by $x$. By $D_{2n}$ we denote the dihedral
group of order $2n$ and by $Q_{4n}$ the quaternion
group of order $4n$. The following finite groups
will  play an important role in the paper:
    $$\matriz{{lll}
    D_{2^{n+2}}^{+} &=&\GEN{a}_{2^{n+1}}\rtimes \GEN{b}_2, \text{ with } ba=a^{2^n + 1}b;\\
    D_{2^{n+2}}^{-} &=&\GEN{a}_{2^{n+1}}\rtimes \GEN{b}_2, \text{ with } ba=a^{2^n - 1}b; \\
    \mathcal{D}&=&\GEN{a,b,c | ca=ac,\; cb=bc,\; a^{2}= b^{2}= c^{4}=1,\; ba=c^{2}ab}; \\
    \mathcal{D}^{+}&=&\GEN{a,b,c | ca=ac,\; cb=bc,\; a^{4}=b^{2}=c^{4}=1,\; ba=ca^{3}b}.}$$
We also need the central product $D_8YQ_{2^n}$ of $D_8$ and $Q_{2^n}$, i.e. $D_8YQ_{2^n}=(D_8\times Q_{2^n})/\GEN{(z_1,z_2)}$, where $z_1$ and $z_2$ are generators of the center of $D_8$ and $Q_{2^n}$ respectively.
Recall that a non-abelian group $G$ is said to be Hamiltonian if every subgroup of $G$ is normal in $G$. The finite Hamiltonian groups are the groups of the form $Q_8\times C_2^n \times A$ with $A$ a finite abelian group of odd order \cite[5.3.7]{R}.

If $F$ is a field and $a,b$ are non-zero elements of $F$ then $\quat{a,b}{F}$ denotes the quaternion algebra $F[i,j|i^2=a,j^2=b,ji=-ij]$. The Hamiltonian quaternion algebra $\quat{-1,-1}{F}$ is denoted $\HQ(F)$. Recall that a quaternion algebra $\quat{a,b}{F}$ over a number field $F$ is totally definite if $F$ is a totally real field such that $a,b$ are totally negative (i.e. $\sigma(F)\subseteq \R$ and $\sigma(a)$ and $\sigma(b)$ are negative for every homomorphism $\sigma:F\rightarrow \C$).

Let $A$ be a finite dimensional semisimple rational algebra and $R$ a order in $A$.
Hence $A\cong A_1\times \dots \times A_n$ with $A_1,\dots,A_n$ simple algebras. Such an expression is called the Wedderburn decomposition of $A$ and the factors $A_i$ are called the simple components of $A$.
The following Wedderburn decompositions can be found in \cite[p.161-163]{CR}, \cite[Lemma~20.4]{seh-book2} or \cite{JL1}:
    \begin{equation}\label{WedderburnDecomposition}
    \matriz{{rcl}
    \Q C_n &\cong& \oplus_{d|n} \Q(\zeta_d), \\
    \Q D_{2n} &\cong& 2 \Q (D_{2n}/D'_{2n}) \oplus \oplus_{2< d|n} M_2(\Q(\zeta_d+\zeta_d\inv)), \\
    \Q Q_{2^n} &\cong& \Q D_{2^{n-1}} \oplus \HQ(\Q(\zeta_{2^{n-1}}+\zeta_{2^{n-1}}\inv)),\\
    \Q D_{16}^- &\cong& 4\Q \oplus M_2(\Q) \oplus M_2(\Q(\sqrt{-2})), \\
    \Q D_{16}^+ &\cong& 4\Q\oplus 2\Q(i) \oplus M_2(\Q(i)),\\
    \Q \mathcal{D} &\cong& 8\Q\oplus M_2(\Q(i)),\\
    \Q \mathcal{D}^+ &\cong& 4\Q\oplus 2\Q(i)\oplus 2M_2(\Q)\oplus 2M_2(\Q(i)), \\
    \Q(D_8YQ_8) & \cong & 16\Q \oplus M_2(\HQ(\Q)).}
    \end{equation}

By an order in $A$ we mean a $\Z$-order in $A$, i.e. a subring of $A$ with finitely generated underlying additive group and containing a basis of $A$ over $\Q$.
It is well known that if $R$ and $S$ are orders in $A$ then $R^*\cap S^*$ has finite index in both $R^*$ and $S^*$ (see e.g. \cite[Lemmas~4.2~and~4.6]{seh-book2}).

We say that $A$ is {\em virtually central (VC)} if the center of $R^*$, for $R$ an order in $A$, has finite index in $R^*$. This definition does not depend on the choice of the order.
If $A$ is simple then $A$ is VC if and only if it is either a field or a totally definite quaternion algebra \cite[Lemma~21.3]{seh-book2}. Therefore, in general, $A$ is VC if and only if all its simple components are fields or totally definite quaternion algebras.


We now recall some elementary properties of subgroup separability.
It is easy to see that abelian groups are subgroup separable and
that the class of subgroup separable groups is closed for
subgroups. Moreover, if $\Lambda$ is a subgroup of
finite index in $\Gamma$ and $\Lambda$ is subgroup separable then
$\Gamma$ is subgroup separable. This implies that if $R$ and $S$
are orders in a finite dimensional semisimple rational algebra
then $R^*$ is subgroup separable if and only if so is $S^*$. If
$\Gamma$ is a subgroup separable group and $\Omega$ is a finitely
generated abelian group then it is known that $\Gamma\times
\Omega$ is subgroup separable (see e.g. \cite[Lemma~4]{MR}).
However, subgroup separability fails to be preserved by many
natural operations. For instance, if $F$ is a non-abelian free
group then $F\times F$ is not subgroup separable. So the class of
subgroup separable groups is not closed under direct products.

\bigskip

The following proposition links subgroup separability of $\Z G^*$ with the Wedderburn decomposition of $\Q G$.

\begin{proposition}\label{Proposition}
Let $G$ be a finite group. Then $\Z G^*$ is subgroup separable if and only if one of the following conditions holds:
 \begin{enumerate}
 \item $\Q G$ is VC.
 \item $\Q G$ has exactly one non-VC simple component $A$ and if $R$ is an (any) order in $A$ then $R^*$ is subgroup separable.
 \end{enumerate}
 \end{proposition}

\begin{proof}
Let $\Q G = A_1\times \cdots \times A_n$ be the Wedderburn decomposition of $\Q G$ and let $R_i$ be an order in $A_i$. As both $\Z G$ and $R=R_1\times \dots \times R_n$ are orders in $\Q G$, it follows that $\Z G^*$ is subgroup separable if and only if so is $R^*$.

If condition 1 holds then $R^*$ contains a finitely generated abelian  subgroup of finite index.  If condition 2 holds and $A_1$ is the only non-VC simple component of $\Q G$ then $R_1^*$ is subgroup separable and $R_2^*\times \dots \times R_n^*$ has a finitely generated abelian subgroup $H$ of finite index. Thus $R_1^*\times H$ is a subgroup separable subgroup of finite index in $R^*$. In both cases $R^*$ is subgroup separable and hence so is $\Z G^*$.

Conversely, assume that $\Z G^*$ is subgroup separable. Then $R^*$ is subgroup separable and hence so is each $R_i^*$.
By Tits Alternative each $R_i^*$ is either virtually  solvable or contains a non-abelian free group. Since the direct product of two non-abelian free groups is not subgroup separable, the number of $R_i^*$'s which are not virtually solvable is at most 1. If $R_i^*$ is virtually solvable then $A_i$ is VC \cite[Theorem 2]{kleinert}. Therefore $\Q G$ has at most one non-VC simple component.
\end{proof}

Observe that the class of finite groups $G$ such that $\Z G^*$ is subgroup separable is closed under subgroups and epimorphic images.
The first is an obvious consequence of the fact that the class of subgroup separable groups is closed under subgroups and the second is a consequence of Proposition~\ref{Proposition}. We will use this throughout without specific mention.

Let $A=M_n(D)$ with $D$ a finite dimensional division rational
algebra and $R$ an order in $D$. Then the group of units of an
order in $A$ is subgroup separable if and only if so is
$\GL_n(R)$. Moreover, $\GL_n(R)$ contains a subgroup of finite
index of the form $H\times K$ where $H$ is a subgroup of finite
index in the center of $R^*$ and $K$ is a subgroup of finite index
in $\SL_n(R)$. Therefore $\GL_n(R)$ is subgroup separable if and
only if so is $\SL_n(R)$. This and Proposition~\ref{Proposition}
imply that it is relevant to consider the problem of when
$\SL_n(R)$ is subgroup separable for $R$ an order in a finite
dimensional rational division algebra $D$. This is, in general, a
difficult problem with many known negative results and few
positive ones. Most of the negative results follow from the fact
that if $\SL_n(R)$ is subgroup separable then it does not have the
Congruence Subgroup Property. In particular, if $\SL_n(R)$ is
subgroup separable then $n\leq 2$ and if $n=2$ then $D$ is either
$\Q$, an imaginary quadratic extension of the $\Q$ or a totally
definite quaternion algebra over $\Q$ (see Main Theorem on page 74
in \cite{ra2} and also 5.6 of \cite{PR} for a short proof written
for fields that is valid for division algebras as well). This proves
the following lemma.

\begin{lemma}\label{ComponenteNoVC}
Let $G$ be a finite group such that $\Z G^*$ is subgroup separable and $A$ a non-VC simple component of $\Q G$. Then $A$ is either a division algebra or isomorphic to $M_2(D)$ with $D$  either $\Q$, an imaginary quadratic extension of $\Q$ or a totally definite quaternion algebra over $\Q$.
\end{lemma}

We say that a group $G$ is decomposable if it is the direct product of two non-trivial subgroups. Otherwise we say that $G$ is indecomposable. Proposition~\ref{Proposition} and Lemma~\ref{ComponenteNoVC} implies strong conditions for finite decomposable groups $G$ such that $\Z G^*$ is subgroup separable.

\begin{lemma}\label{Indecomposable}
If $G$ is a finite non-abelian decomposable group such that $\Z G^*$ is subgroup separable then one of the following conditions holds:
\begin{enumerate}
\item $G\cong Q_8\times C_2^k$ for some $k\ge 1$.
\item $G\cong Q_8\times C_n$, with $n$ either 3, 4 or prime satisfying $n\equiv -1 \mod(8)$.
\end{enumerate}
\end{lemma}

\begin{proof}
Assume that $G=H\times K$ with $H$ non-trivial and $K$ non-abelian. We claim that $K$ is Hamiltonian. Otherwise one of the simple components of $\Q K$ is not a division algebra and so, by Lemma~\ref{ComponenteNoVC}, it is of the form $M_2(D)$ for $D$ a division algebra. As $\Q H$ has at least two simple components, $\Q G$ has at least two simple components which are not division algebras, and hence they are not VC. This contradicts Proposition~\ref{Proposition} and finishes the proof of the claim.

If $H$ is non-abelian then it is also Hamiltonian, by the previous paragraph. Then $G$ contains a subgroup isomorphic to $Q_8\times Q_8$.
As $\Q(Q_8\times Q_8)$ has a simple component isomorphic to $\HQ(\Q)\otimes_{\Q}\HQ(\Q)\cong M_4(\Q)$, the group $\Z(Q_8\times Q_8)^*$ is not subgroup separable, by Lemma~\ref{ComponenteNoVC}. This yields to a contradiction. Therefore $H$ is abelian.

Let $n>1$. Then $\Q(Q_8\times C_n)$ has a simple component isomorphic to $\HQ(\Q(\zeta_d))$, for every divisor $d$ of $n$. This algebra is VC if and only if $d=1$ or $2$. Therefore, if $\Z(Q_8\times C_n)^*$ is subgroup separable then $n$ has at most one divisor different from $1$ or $2$ and hence $n$ is either 4 or prime. The same argument shows that if $\Z(\Q_8\times C_n \times C_m)^*$ is subgroup separable with $n$ and $m$ different of 1 then $n=m=2$.

This implies that $K\cong Q_8\times A$ with $A$ an elementary abelian 2-group, and $H$ is either elementary abelian 2-group or cyclic of order 4 or prime. Moreover, if $A\ne 1$ then $H$ is elementary abelian 2-group. Thus either $G$ satisfies condition 2 or $G=Q_8\times C_n$ with $n=4$ or an odd prime. Assume that $G=Q_8\times C_n$ with $n$ odd prime. Then one of the simple components of $\Q G$ is isomorphic to $\HQ(\Q(\zeta_n))$. If moreover $n\not\equiv -1 \mod (8)$ then $\HQ(\Q(\zeta_n))\cong M_2(\Q(\zeta_n))$ (see e.g. the paragraph below \cite[Proposition~2.11]{Lam}). By Lemma~\ref{ComponenteNoVC}, $n-1=[\Q(\zeta_n):\Q]\le 2$ and hence $n=3$. This finishes the proof.
\end{proof}

By Lemma~\ref{ComponenteNoVC}, if $\Z G^*$ is subgroup separable then every simple component of $\Q G$ is either a division algebra or a two-by-two matrix ring over a division algebra. The simple components of this form, for $G$ a nilpotent group, have been classified in \cite{JL1}. We will use this in our next lemma.

\begin{lemma}\label{SimpleComponents}
Let $G$ be a non-abelian nilpotent finite group. Let $e$ be a primitive central idempotent of $\Q G$ such that $(\Q G)e$ is not abelian. If $\Z G^*$ is subgroup separable then one of the following facts holds:
 \begin{enumerate}
 \item $Ge\cong Q_8\times C_3$ and $(\Q G)e=M_2(\Q (\sqrt{-3}))$.
 \item $Ge\cong Q_8\times C_p$ with $p$ prime satisfying $p\equiv -1\; mod(8)$ and $(\Q G)e=\HQ\left(\Q\left(\zeta_p\right)\right)$.
\item $Ge\cong D_8$ and $(\Q G)e=M_2(\Q)$.
\item $Ge\cong Q_8$ and $(\Q G)e=\HQ(\Q)$.
\item $Ge\cong Q_{16}$ and $(\Q G)e=\HQ(\Q(\sqrt{2}))$.
\item $Ge\cong D_{16}^+$ and $(\Q G)e=M_2(\Q(i))$.
\item $Ge\cong D$ and $(\Q G)e=M_2(\Q(i))$.
\item $Ge\cong D_8YQ_8$ and $(\Q G)e=M_2(\HQ(\Q))$.
 \end{enumerate}
\end{lemma}

\begin{proof}
As $Ge$ is an epimorphic image of $G$, $\Z(Ge)^*$ is subgroup separable and $(\Q G)e$ is a simple component of $\Q G$ isomorphic to a simple component of $\Q(Ge)$. We separate cases depending on whether $Ge$ is a $p$-group or not. The $p$-group case is the most involved and it is split depending on whether $\Q Ge$ is a division algebra, a matrix algebra over a field or a matrix algebra over a non-commutative division algebra.

If $Ge$ is not a $p$-group, for some $p$ then, by Lemma~\ref{Indecomposable}, $Ge \cong Q_8\times C_p$ with $p$ prime and either $p=3$ or $p \equiv -1 \mod (8)$. In the first case $(\Q G)e \cong M_2(\Q(\sqrt{-3}))$ and in the second case $(\Q G)e\cong \HQ\left(\Q\left(\zeta_p\right)\right)$. Therefore if $Ge$ is not a $p$-group then either condition 1 or 2 holds.

Assume otherwise that $Ge$ is a $p$-group for some prime $p$. If $p$ is odd then, by a well known result of Roquette \cite{Roquette}, $(\Q G)e$ is an $n\times n$ matrix algebra over a field, for $n$ a power of $p$, contradicting Lemma~\ref{ComponenteNoVC}. Thus $Ge$ is a $2$-group and, by Lemma~\ref{ComponenteNoVC}, $(\Q G)e$ is either a division algebra or a 2-by-2 matrix algebra over a division algebra. Then $Ge$ and $(\Q G)e$ satisfy one of the conditions  of \cite[Theorem~2.2]{JL1}.

If $(\Q G)e$ is a division algebra then $Ge$ is isomorphic to $Q_{2^n}$ and $(\Q G)e=\HQ(\Q(\zeta_{2^{n-1}}+\zeta_{2^{n-1}}^{-1}))$. Then $D_{2^{n-1}}$ is an epimorphic image  of $Ge$. Hence $\Q G$ has a simple component isomorphic to $M_2(\Q(\zeta_{2^{n-1}}+\zeta_{2^{n-1}}^{-1}))$. (See the Wedderburn decompositions in (\ref{WedderburnDecomposition}). By Lemma~\ref{ComponenteNoVC} it follows that $\Q(\zeta_{2^{n-1}}+\zeta_{2^{n-1}}^{-1})=\Q$ and thus $n\leq 4$. Therefore, in this case either condition $4$ or $5$ holds.

Assume that $(\Q G)e\cong M_2(F)$ with $F$ a field. By \cite[Theorem 2.2]{JL1},  $Ge$ is isomorphic to one of the following groups: $D_8,D_{16},D_{16}^+,D_{16}^-, \mathcal{D}$ or $\mathcal{D}^+$. By inspection of the Wedderburn decomposition of these groups (\ref{WedderburnDecomposition}) we observe that $\Q D_{16}, \Q D_{16}^-, \Q \mathcal{D}$ and $\Q \mathcal{D}^+$ have at least two non-VC simple components, yielding to a contradiction with Proposition~\ref{Proposition}. We conclude that $Ge$ is either $D_8$, $D_{16}^+$ or $\mathcal{D}$. Then either condition 3, 6 or 7 holds.

 It remains to consider the case when $(\Q G)e\cong M_2(D)$ with $D$ a non-commutative division algebra. By \cite[Theorem 2.2]{JL1}, $D\cong \HQ(\Q(\zeta_{2^{n-1}}+\zeta_{2^{n-1}}\inv))$ and $G=\GEN{H,g}$, with $H$ a subgroup of index 2 in $G$, and $H$ contains a non-trivial normal subgroup $N$ such that $N\cap N^g = 1$ and $H/N\cong Q_{2^n}$. (Observe that case (3.a) in loc. cit. is in fact contained in case (3.b) because if $D_8=\GEN{a,b}$, with $a$ of order $4$, $H=\GEN{b,Q_{2^n}}$ and $N=\GEN{b}$ then $H$ and $N$ satisfy  condition (3.b) for $Ge=D_8YQ_{2^n}$.) By Lemma~\ref{ComponenteNoVC}, $\Q(\zeta_{2^{n-1}}+\zeta_{2^{n-1}}\inv)=\Q$ and hence $n=3$. Since $N\cap N^g=\{1\}$   and $N$ is a normal subgroup of $H$, $\GEN{N,N^g}=N\times N^g\subseteq H$ and $(N\times N^g)/N$ is a non-trivial subgroup of $H/N\cong Q_8$. Thus $N$ is isomorphic to a subgroup of $Q_8$ and therefore its order is either 2, 4 or 8. If $|N|=8$ then $N\times N^g$ is isomorphic to $Q_8\times Q_8$ and its rational group algebra contains a simple component isomorphic to $M_4(\Q)$ yielding to a contradiction with Lemma~\ref{Proposition}. Thus $N$ has order 2 or 4.
We claim that $N$ has order $2$. Otherwise $N$ is generated by an element of order $4$, since so is every subgroup of order 4 of $Q_8$. Then $H=\GEN{x,a,b}$ with $N=\GEN{x}_4$, $a=x^g$, $b^2N=a^2N$ and $a^bN=a\inv N$. As $N$ is normal in $H$ and $a$ is not central in $H$ we have $x^b=x\inv$ and $a^b=a\inv$. Moreover $b^2=a^2x^i$, with $i=0,1,2$ or $3$. As $a^2$ and $b^2$ are central in $H$ and $x$ is not, necessarily $i=0$ or $2$. In both cases $H/\GEN{b^2}$ is isomorphic to $D_8\times C_2$. This yields to a contradiction with Lemma~\ref{Indecomposable}.

Thus $N=\GEN{x}_2$, $|H|=16$, $H/N\cong Q_8$ and $G=\GEN{H,g}$ with $x^g\ne x$. This implies that $H$ is isomorphic to either $Q_8\times C_2$ or $\GEN{c}_4\rtimes \GEN{d}_4$, with $dc=c^3d$. Assume that $H$ is as in the second case. Then $H$ has 3 elements of order 2, namely $c^2, d^2$ and $c^2d^2$. Notice that $c^2$ is the only element of order 2 in $H'=\GEN{c^2}$ and $c^2d^2$ is the only non-square element of order 2 of $H$. This proves that $c^2,d^2$ and $c^2d_2$ are invariant by any automorphism of $H$. As $x$ is an element of order 2 of $H$ it follows that $x^g= x$, a contradiction.
Therefore $H\cong Q_8\times C_2$. This implies that for every $a,b\in H$ with $(a,b)\ne 1$ we have $H=\GEN{a,b}\times \GEN{x}$, $\GEN{a,b}\cong Q_8$ and $x^g=a^2x$.

In the remainder of the proof we will use that $D_8$ is not an epimorphic image of $G$. Otherwise $M_2(\Q)$ is a simple quotient of $\Q G$ and by assumption $M_2(\HQ (\Q))$ is another simple quotient of $\Q G$ yielding to a contradiction with Proposition~\ref{Proposition}.

We claim that the order of $g$ is either 2 or 4 and in fact we may assume that it is 4.
If $g$ is of order 8 then we may assume that $g^2=a$. Thus, $x^g=g^4x$ and therefore $\GEN{g,x}$ is a normal subgroup of $G$ isomorphic to $D_{16}^+$. Then $g^b=g^ix^j$ with $i=\pm 1$ or $\pm 3$ and $j=0$ or 1. Also $g^{-2}=a^{-1}=a^b=(g^2)^b=(g^ix^j)^2$. If $j=0$ then $g^{-2}=g^{2i}$. Therefore $i\equiv -1 \mod 4$ and thus $G/\GEN{a^2,x}$ is isomorphic to $D_8$, a contradiction. So $j=1$ and $g^{-2}=g^{6i}$. Therefore $i=1$ or $-3$. In this case $G/\GEN{xg^2}$ is isomorphic to $D_8$, again a contradiction. Then the order of $g$ is 2 or 4. If the order of $g$ is 2 then $gx$ has order $4$. Hence, we may assume that $g$ has order 4 as desired.

Thus in the remainder of the proof we assume that $g$ has order $4$. Then $g^2$ is an element of order 2 of $H$ which commutes with $g$ and hence $g^2=a^2$. The group $H$ has three abelian subgroups of order $8$, namely, $\GEN{a,x}, \GEN{b,x}$ and $\GEN{ab,x}$. If any of these groups is not fixed by the action of $g$ then we may assume that $a^g=b$ (changing $b$ by $a^g$ if needed). Then $(g,a)=b\inv a=ab$ and thus the quotient $G/\GEN{a^2,x}$ is a nonabelian group of order 8 generated by two elements of order 2. Hence $G/\GEN{a^2,x}\cong D_8$, a contradiction. So the action of $g$  fixes the three subgroups of order 8 in $H$. If $\GEN{a}$ is not normal in $G$ then $a^g=ax$ or $a^g=a^{-1}x$. Then $a=(a^g)^g$ is equal to either $(ax)^g=axa^2x=a^{-1}$ or $(a^{-1}x)^g=axa^2x=a^{-1}$, a contradiction. This proves that every cyclic subgroup of order 4 of $H$ is normal in $G$. Therefore, if $(a,g)\ne 1$ then $a^g=a\inv$ and hence $(ax)^g = ax$. Thus replacing $a$ by $ax$ if needed we may assume that $(g,a)=1$ and similarly, one may assume that $(g,b)=1$. Hence $G=\GEN{g,x}Y\GEN{a,b}=D_8YQ_8$ which finishes the proof of the lemma.
\end{proof}

\begin{theorem}\label{SS}
Let $G$ be a nonabelian finite  group such that $\Z G^*$ is subgroup separable. Then $G$ is either abelian or isomorphic to one of the
following groups:
$$D_6, D_8, Q_{12}, C_4\rtimes C_4, \mathcal{D}, D_{16}^+, Q_{16}, Q_8\times C_3, Q_8\times C_4, D_8YQ_8,$$
$$Q_8\times C_2^n \text{ with } n\ge 0, \text{ or }$$
$$Q_8\times C_p \mbox{ with $p$ prime and }p\equiv -1 \mod (8).$$
\end{theorem}

\begin{proof}
If $G$ is decomposable then, by Lemma~\ref{Indecomposable}, $G$ is isomorphic to either $Q_8\times C_2^n \text{ (with } n\ge 1),  Q_8\times C_3$, $Q_8\times C_4$ or $Q_8\times C_p$ with $p$ prime and $p\equiv -1 \mod 8$.
 So in the remainder of the proof we assume that $G$ is indecomposable. We consider cases depending on whether $G$ is nilpotent or not.

\emph{Assume that $G$ is nilpotent}.
Then, $G$ is a $p$-group, because it is indecomposable and, by Lemma~\ref{SimpleComponents}, $G$ is a $2$-group. Moreover, for every primitive central idempotent $e$ of $\Q G$ such that $Ge$ is not abelian, one of the conditions 3-8 of Lemma~\ref{SimpleComponents} holds. If $G$ is Hamiltonian then $G$ is isomorphic to $Q_8$.
Assume that $G$ is not Hamiltonian. If  $Q_{16}$ is not an epimorphic image of $G$ then, by Lemma~\ref{SimpleComponents}, every non-commutative simple quotient of $\Q G$ is isomorphic to either $M_2(\Q)$, $\HQ(\Q)$, $M_2(\Q(i))$ or $M_2(\HQ(\Q))$ and only one simple component is not a division algebra, by Proposition~\ref{Proposition}. The non-abelian finite groups $G$ satisfying this condition have been classified in \cite[Theorem 1]{JL2}. Using this result we deduce that $G$ is isomorphic to either $D_8,  C_4\rtimes C_4, \mathcal{D}, D_{16}^+$ or $D_8YQ_8$.

Assume otherwise that $Q_{16}$ is an epimorphic image of $G$. Then $D_8$ is also an epimorphic image of $G$ and therefore $M_2(\Q)$ is isomorphic to a simple component of $\Q G$. Then the remaining simple components of $\Q G$ are division algebras, by Proposition~\ref{Proposition}. By Lemma~\ref{SimpleComponents}, every simple quotient of $\Q G$ is isomorphic to either $M_2(\Q), \HQ (\Q)$ or $\HQ (\Q(\sqrt{2}))$. Then $G$ satisfies condition $(3)$ of \cite[Theorem~1.3]{JR}. Thus $G$ is one of the groups (a)-(g) listed in that result, because $G$ is non-abelian indecomposable 2-group and the groups (h) and (i) in the list are not 2-groups. The groups (a)-(f) have exponent 4, while the exponent of $G$ is at least 8 because $Q_{16}$ is an epimorphic image of $G$. Thus $G$ is isomorphic to the group $H_n$ given by the presentation $\GEN{x,y_1,\dots,y_n|x^4=x^2y_i^4=y_i^2[x,y_i]=[y_i,y_j]=1}$ for some $n\ge 1$. As  $H_n/\GEN{y_2^2,\dots,y_n^2}\cong Q_{16}\times C_2^{n-1}$, does not satisfies the conditions of Lemma~\ref{Indecomposable} if $n>1$, we deduce that $n=1$. We conclude that $G\cong Q_{16}$. This finishes the proof for the nilpotent case.

\emph{Assume that $G$ is non-nilpotent}.
By Proposition~\ref{Proposition}, every simple component of $\Q G$ is either a division algebra or a two-by-two matrix ring over a division algebra. In other words the reduced degree over $\Q$ of each irreducible character of $G$ is either 1 or 2. This implies that $G$ contains a nilpotent subgroup of index 2, by \cite{GH,GH2}. Hence $G=N_{2'}\rtimes G_2$ where $N_{2'}$ is a nilpotent $2'$-group and $G_2$ is a 2-group such that $N_2=\Cen_{G_2}(N_{2'})$ has index 2 in $G_2$. Therefore, there is a non-trivial automorphism $\sigma$ of $N_{2'}$ of order 2, such that for every $x\in G_2$, the action $\varphi_x$ of $x$ on $N_{2'}$ by conjugation is trivial if $x\in N_2$ and otherwise $\varphi_x=\sigma$.

We claim that $G_2$ is cyclic. Assume first that $G_2$ is abelian and write $G_2=\GEN{x_1}_{n_1}\times \dots \times \GEN{x_k}_{n_k}$ with $2\le n_1\le n_2\le \dots \le n_k$. Let $i$ be minimum with $x_i\not\in N_2$. By replacing, $x_j$ by $x_jx_i$, for each $j>i$ such that $x_j\not\in N$, we may assume that $x_j\in N_2$ for every $j\ne i$. Then $G=(N_{2'}\rtimes \GEN{x_i})\times \prod_{j\ne i} \GEN{x_j}$. As, by assumption, $G$ is indecomposable we deduce that $k=1$, as wanted. Assume otherwise that $G_2$ is non-abelian.  By the nilpotent case $G_2$ is one of the 2-groups listed in the theorem. On the other hand $G_2'$ is a normal subgroup of $G$ and $G/G_2' \cong N_{2'}\rtimes (G_2/G_2')$. By the abelian case, $G_2/G_2'$ is cyclic. This yields to a contradiction, since none of the 2-groups listed in the theorem satisfies this condition.

Hence $G_2=\GEN{x}$ for some $x$, of order $2^n$, say. Now we claim that every subgroup of $N_{2'}$ is normal in $G$. Otherwise there is $a\in N_{2'}$ of order $q$, an odd prime power, such that $b=a^x\not \in \GEN{a}$. This implies that $\GEN{ab,x^2}$ is contained in the center of $G$ and $\GEN{a,b,x}/\GEN{ab,x^2}$ is isomorphic to $D_{2q_1}$, for $q_1$ a divisor of $q$ different than $1$. However $\Q D_{2q_1}$ has a simple component isomorphic to $M_2(\Q(\zeta_{q_1}+\zeta_{q_1}\inv))$. This implies that $\Q(\zeta_{q_1}+\zeta_{q_1}\inv)=\Q$ and hence $q_1=3$. Thus $a^3=(ab)^i=a^ib^i$ for some integer $i$. Therefore $b^i=a^{3-i}$. As $b\not\in \GEN{a}$, we have $i=3m$ for some $m$. Then $a^{3(1-m)}=b^{3m}$. As $a$ and $b$ have the same order, $m$ is coprime with $3$. Thus $b^3\in \GEN{a^3}$. This implies that $\GEN{a^3,x^2}$ is normal in $G$ and $H=\GEN{a,b,x}/\GEN{a^3,x^2}\cong (\GEN{\overline{a}}_3\times \GEN{\overline{b}}_3) \times \GEN{\overline{x}}_2$. The Wedderburn decomposition of $\Q H$ is
\begin{equation*}
\Q H=2\Q \oplus2\Q(\sqrt{-3})\oplus M_2(\Q)\oplus M_2(\Q(\sqrt{-3})).
\end{equation*}
By Proposition~\ref{Proposition}, $\Z H^*$ is not subgroup separable, a contradiction. This finishes the proof of the claim.

Thus every subgroup of $N_{2'}$ is normal in $G$. Therefore, if $a\in N_{2'}$ is an element of order $q$ non-commuting with $x$, then $\GEN{a,x}/\GEN{x^2}\cong D_{2q}$. As in the previous paragraph this implies that $q=3$. Using that $G$ is indecomposable it is now easy to prove that $N_{2'}=C_3$.
Therefore $G=C_3\rtimes C_{2^n}$ with $a^x=a^{-1}$. If $n\geq 3$ then $K=C_3\rtimes C_8$ is isomorphic to an epimorphic image of $G$.
The Wedderburn decomposition of $\Q K$ is
\begin{equation*}
\Q K=2\Q \oplus \Q(i)\oplus\Q(\zeta_8)\oplus M_2(\Q)\oplus \left(\frac{-1,-3}{\Q}\right)\oplus \left(\frac{i,-3}{\Q(i)}\right).
\end{equation*}
By Proposition~\ref{Proposition}, $\Z K^*$ is not subgroup separable, a contradiction. Therefore $G$ is isomorphic to either $C_3\rtimes C_2\cong D_6$ or $C_3\rtimes C_4=Q_{12}$ which finishes the proof of the theorem.
\end{proof}

To obtain a complete classification of the finite groups $G$ such that $\Z G^*$ is subgroup separable one should decide which of the groups appearing in Theorem~\ref{SS} satisfy the conditions of Proposition~\ref{Proposition}. If $G=Q_8\times C_2^n$, with $n\ge 0$, then $\Z G^*$ is finite and hence $\Z G^*$ is subgroup separable. For the remaining groups in Theorem~\ref{SS}, $\Q G$ has precisely one non-VC component. The following table classify the groups appearing in Theorem~\ref{SS}, other than $Q_8\times C_2$, according to the non-VC component $A$. The third
column contains an order $R$ in the non-VC component.
    $$\matriz{{ccc}
    G & A & R \\\hline \\
    D_6, D_8, C_4\rtimes C_4, Q_{16} & M_2\left(\Q\right) & M_2\left(\Z\right) \\
    Q_8\times C_3 & M_2\left(\Q\left(\sqrt{-3}\right)\right) & M_2\left(\Z\left[\sqrt{-3}\right]\right) \\
    Q_8\times C_4, \mathcal{D}, D_{16}^+ & M_2\left(\Q\left(i\right)\right) & M_2\left(\Z[i]\right) \\
    D_8YQ_8 & M_2\left(\HQ\left(\Q\right)\right) & M_2\left(\HQ\left(\Z\right)\right)\\
    Q_8\times C_p, \mbox{ with } p \text{ prime and } p\equiv -1 \mod
    \left(8\right) & \HQ\left(\Q\left(\zeta_p\right)\right) & \HQ\left(\Z\left[\zeta_p\right]\right) \\\\\hline}$$
Let $G$ be
one of the groups in the previous table and let $R$ be the order displayed in the
third column of the table. By
Proposition~\ref{WedderburnDecomposition}, $\Z G^*$ is subgroup
separable if and only if $R^*$ is subgroup separable. This has
been settled for the groups in the first three rows. Indeed, it is
well known that $\GL_2(\Z)$ contains a non-abelian free subgroup
of finite index and it has been proved recently that
$\GL_2(\Z[\sqrt{-3}])$ and $\GL_2(\Z[i])$ are subgroup separable
(see \cite[Theorem~3.4]{LR}). So we have the following positive
result.

\begin{theorem}
If $G$ is one of the following groups
$$D_6, D_8, Q_{12}, C_4\rtimes C_4, \mathcal{D}, D_{16}^+, Q_{16}, Q_8\times C_3, Q_8\times C_4 \text{ or } Q_8\times C_2^n \text{ (with } n\ge 0) $$
then $\Z G^*$ is subgroup separable.
\end{theorem}

To decide whether $\Z G^*$ is subgroup separable or not for $G$
one of the groups in the last two rows of the table one should
decide whether $R^*$ is subgroup separable. A presentation by
generators and relations for $\SL_2(\HQ(\Q))$ has been obtained in
\cite{ALM}. However the subgroup separability question for this
groups does not seem to follow from the presentation.  As far as
we know there is very little known about the structure of the
group of units of $\HQ(\Z[\zeta_p])$, for $p$ prime with $p\equiv
-1 \mod (8)$ and it is not known whether this group is subgroup
separable or not.

Thus to complete the classification of finite
groups $G$ with $\Z G^*$ subgroup separable it remains to decide
if $\GL_2\left(\HQ(\Z)\right)$ is subgroup separable and for which prime
integers $p$ with $p\equiv -1 \mod (8)$, the group of units of
$\HQ\left(\Z\left[\zeta_p\right]\right)$ is subgroup separable.
In fact $\GL_2\left(\HQ(\Z)\right)$ is subgroup separable if and only if so is $\SL_2\left(\HQ(\Z)\right)$. Similarly,
$\HQ\left(\Z\left[\zeta_p\right]\right)^*$ is subgroup separable if and only if so is $\SL_1\left(\HQ\left(\Z\left[\zeta_p\right]\right)\right)$.

A presentation by generators and relations for
$\SL_2\left(\HQ\left(\Z\right)\right)$ has been obtained in
\cite{ALM}. Unfortunately the subgroup separability question for
this groups does not seem to follow from the presentation. Note
that $\SL_2\left(\HQ\left(\Z\right)\right)$ does not posses the
congruence subgroup property, since it contains a subgroup of
finite index that maps onto a free non-abelian group. However, it
is not known whether failure of the congruence subgroup property
implies subgroup separability for arithmetic groups (virtually
indecomposable in direct products).

In  the remaining cases,
$\SL_1\left(\HQ\left(\Z\left[\zeta_p\right]\right)\right)$ with $p$ prime with $p\equiv -1 \mod
(8)$, the congruence subgroup property is unknown and the structure
of the group not-understood.

\end{document}